\documentclass[10pt]{amsart}

\usepackage{amsmath}
\usepackage{amssymb}
\usepackage{amsfonts}
\usepackage{amsthm}
\usepackage{enumerate}
\usepackage{hyperref}
\usepackage{color}

\theoremstyle{plain}
\newtheorem{thm}{Theorem}[section]

\newtheorem{prop}[thm]{Proposition}

\newtheorem{cor}[thm]{Corollary}

\theoremstyle{definition}
\newtheorem{dfn}[thm]{Definition}

\newtheorem{nota}[thm]{Notation}

\newtheorem{exmp}[thm]{Example}

\newtheorem{rem}[thm]{Remark}

\newtheorem{dfns-rems}[thm]{Definitions and Remarks}
\newtheorem{notas-rems}[thm]{Notations and Remarks}
\newtheorem{exmps-rems}[thm]{Examples and Remarks}

\begin{document}


\title[Linear Recurrence Relation Homomorphisms]{Some results about
  Linear Recurrence Relation Homomorphisms} 


\author[A. Laugier]{Alexandre Laugier}

\address{Lyc{\'e}e professionnel Tristan Corbi{\`e}re, 16 rue de
  Kerveguen - BP 17149 - 29671  Morlaix Cedex, France}

\email{laugier.alexandre@orange.fr}


\author[M. P. Saikia]{Manjil P. Saikia}

\address{Diploma Student, Mathematics Group, The Abdus Salam International Centre for Theoretical Physics, Strada Costiera-11, Miramare, I-34151, Trieste, Italy}

\email{manjil@gonitsora.com, msaikia@ictp.it}

\urladdr{http://www.manjilsaikia.in}




\begin{abstract}
In this paper we propose a definition of a recurrence relation
homomorphism and illustrate our definition with a few examples. We
then define the period of a k-th order of linear recurrence relation
and deduce certain preliminary results associated with them. 
\end{abstract}








\maketitle


\section{Introduction and Motivation}
This paper is divided into two sections, in the first section we give
some introductory remarks and set the notation for the rest of the
paper; whereas in the second section we discuss linear recurrence
relation homomorphisms and discuss some preliminary properties of such
homomorphisms. 

We begin with the following definitions from \cite{gandhi} and a few
notations to be used throughout this paper. 

\begin{dfn}
A $k$-th order of recurrence relation on some set $X$ is a function $a\,:\,\mathbb{N} \rightarrow X$ with $a_1,\ldots,a_k$ defined for all $i\geq 0$, $k\geq 1$ and $a_{i+k+1}=f(a_{i+1}, \ldots, a_{i+k})$.
\end{dfn}

\begin{dfn}\label{defrrh}
Let $a_n$ be a $k$-th order recurrence relation on the set $X$ defined by the map $f\,:\, X^k\rightarrow X$ with initial values. A map $\varphi\,:\,X\rightarrow Y$ is said to be a recurrence relation homomorphism on $a$, when there exists $f' \,:\, Y^k \rightarrow Y$ satisfying $\varphi \circ f=f\circ \varphi$.
\end{dfn}

\begin{nota}
$(m,n)$ denotes the gcd of $m$ and $n$ for natural numbers $m$ and $n$.
\end{nota}

\begin{nota}
We denote the set $\{1, 2, \ldots, n\}$ by $[[1,n]]$ for $n\geq 2$.
\end{nota}

\begin{dfn}
A sequence $(b_n)$ is called a strong divisibility sequence if $(b_n, b_m)=b_{(m,n)}$.
\end{dfn}

\begin{dfn}
The Fibonacci sequence $(F_n)$ is defined in the usual way as $F_0=0, F_1=1, F_2=1$ and $F_n=F_{n-1}+F_{n-2}$ for $n\geq 2$.
\end{dfn}

We see now that taking
$k\geq 1$ and defining
the maps $f\,:\,X^k\rightarrow X$, $\varphi\,:\,X\rightarrow Y$ and
$\varphi^{(k)}\,:\,X^k\rightarrow Y^k$ such that
$\varphi^{(k)}(a_i,\ldots,a_{i+k-1})
=(\varphi(a_i),\ldots,\varphi(a_{i+k-1}))$  for $i\geq 1$, and $a$ be a $k^{th}$
order of recurrence relation on $X$ which is defined by the map $f$
(with initial values $a_1,\ldots,a_k$), if there exists
$f'\,:\,Y^k\rightarrow Y$ such that
$$\varphi\circ f(a_i,\ldots,a_{i+k-1})
=f'\circ\varphi^{(k)}(a_i,\ldots,a_{i+k-1})$$ for all $i\geq 1$
and for all $(a_i,\ldots,a_{i+k-1})\in X^k$,
then the diagram
$$
\begin{array}{ccccc}
& X^k &\stackrel{f}{\longrightarrow} & X &\\
\varphi^{(k)} &\downarrow & &\downarrow &\varphi\\
& Y^k &\stackrel{f'}{\longrightarrow} & Y &
\end{array}
$$
commutes. That is $
\varphi\circ f=f'\circ\varphi^{(k)}.
$

So we propose the following alternative definition of a recurrence relation
homomorphism as in Definition \ref{defrrh} which maps set $X$ onto set $Y$

\begin{dfn}
Let $a\,:\,\mathbb{N}\rightarrow X$ be a $k^{th}$ order of
  recurrence relation on some set $X$ such that $a_1,\ldots,a_k$
  defined and for all $i,k\geq 1$,
  $a_{i+k}=f(a_i,\ldots,a_{i+k-1})$ with $f\,:\,X^k\rightarrow
  X$. A map $\varphi\,:\,X\rightarrow Y$ is said to be a recurrence
  relation homomorphism on $a$, when there exists
  $f'\,:\,Y^k\rightarrow Y$ satisfying the commutative relation
  $\varphi\circ f=f'\circ\varphi^{(k)}$.
  \end{dfn}

We shall now give an alternate proof of the following theorem that appears in \cite{gandhi}.

\begin{thm}
Suppose we are given a recurrence relation homomorphism in the above notation, then $b_n=\varphi (a_n)$ is a $k$-th order of recurrence relation.
\end{thm}

\begin{proof}

It suffices to state that, according to our definition, defining the sequence
$b\,:\,\mathbb{N}\rightarrow Y$ by
$b_n=\varphi(a_n)$, we have for the given $k$ initial values, with $i, k\geq 1$,
\begin{align}
b_{i+k} &= \varphi(a_{i+k})=\varphi(f(a_i,\ldots,a_{i+k-1}))
=f'(\varphi^{(k)}(a_i,\ldots,a_{i+k-1}))\nonumber\\
 &= f'(\varphi(a_i),\ldots,\varphi(a_{i+k-1}))
=f'(b_i,\ldots,b_{i+k-1}).\nonumber
\end{align}

\end{proof}

We now illustrate our definition with the following examples.

\begin{exmp}
Let $X$ be the ring of integers. Let $a$ be a
$k^{th}$ order of recurrence relation on $\mathbb{Z}$ defined by the
linear map $f\,:\,\mathbb{Z}^k\rightarrow\mathbb{Z}$ with $k\geq 1$ and
initial values $a_1,\ldots,a_k$
given such that for $i\geq 1$ we have
$$
a_{i+k}=f(a_i,\ldots,a_{i+k-1})={\displaystyle\sum^k_{j=1}}f_j\cdot a_{i+j-1},
$$
with $(f_1,\ldots,f_k)\in\mathbb{Z}^k$. A particular case is when
$f_j=a_j$ with $j=1,\ldots,k$. It can be compared to the relation
$F_{n+m}=F_mF_{n+1}+F_{m-1}F_n$
with $n\in\mathbb{N}$ and $m\in\mathbb{N}^{\star}$.

\end{exmp}

\begin{exmp}

In the following, we use the system of residue
classes of integers modulo $m\geq 1$ given by $[0]_m,\ldots,[m-1]_m$ where the notation $[x]_m$ means the equivalence class
of the integer $x\in[[ 0,m-1]]$ modulo $m\geq 1$
$$
[x]_m=\left\{x+km\,:\,k\in\mathbb{Z}\right\}.
$$
Let us consider the map
$\pi_m\,:\,\mathbb{Z}\rightarrow\mathbb{Z}/m\mathbb{Z}$ with
$m\in\mathbb{N}^{\star}$ defined for $x\in\mathbb{Z}$ by,
$$
\pi_m(x)=[x]_m.
$$
By our construction of $\pi_m$, it is a
surjective morphism of rings. So, we will have for $i, k \geq 1$,
$$
\pi_m(a_{i+k})={\displaystyle\sum^k_{j=1}}[f_j]_m\cdot [a_{i+j-1}]_m.
$$
Therefore ($i, k \geq 1$)
$$
[a_{i+k}]_m={\displaystyle\sum^k_{j=1}}[f_j]_m\cdot [a_{i+j-1}]_m,
$$
and so for $i, k \geq 1$ we have
$$
[a_{i+k}]_m=\psi_f([a_i]_m,\ldots,[a_{i+k-1}]_m).
$$ where
$\psi_f\,:\,(\mathbb{Z}/m\mathbb{Z})^k\rightarrow\mathbb{Z}/m\mathbb{Z}$
is the linear map defined by ($i,k\geq 1$):
$$
\psi_f([a_i]_m,\ldots,[a_{i+k-1}]_m)
={\displaystyle\sum^k_{j=1}}[f_j]_m\cdot [a_{i+j-1}]_m.
$$
So, as we can observe, $[a]_m$ is a $k^{th}$ order of recurrence
relation on the ring $\mathbb{Z}/m\mathbb{Z}$ such that all
$[a_1]_m,\ldots,[a_k]_m$ are defined and for all $i, k\geq 1$, $[a_{i+k}]_m=\psi_f([a_i]_m,\ldots,[a_{i+k-1}]_m)$
with $\psi_f\,:\,(\mathbb{Z}/m\mathbb{Z})^k
\rightarrow\mathbb{Z}/m\mathbb{Z}$.

Moreover, we have ($i, k \geq 1$)
\begin{align}
\psi_f(\pi^{(k)}_m(a_i,\ldots,a_{i+k-1})) &=\psi_f(\pi_m(a_i),\ldots,\pi_m(a_{i+k-1}))\nonumber\\
 &= \psi_f([a_i]_m,\ldots,[a_{i+k-1}]_m)\nonumber\\
 &= [a_{i+k}]_m=\pi_m(a_{i+k})=\pi_m(f(a_i,\ldots,a_{i+k-1})).\nonumber
 \end{align}
Since $f$ is any linear function which maps the set $\mathbb{Z}^k$ onto
the set $\mathbb{Z}$, the diagram
$$
\begin{array}{ccccc}
& \mathbb{Z}^k &\stackrel{f}{\longrightarrow} & \mathbb{Z} &\\
\pi^{(k)}_m &\downarrow & &\downarrow &\pi_m\\
& (\mathbb{Z}/m\mathbb{Z})^k &\stackrel{\psi_f}{\longrightarrow} &
\mathbb{Z}/m\mathbb{Z} &
\end{array}
$$
commutes. That is $
\pi_m\circ f=\psi_f\circ\pi^{(k)}_m.
$
\end{exmp}

\section{Some results on recurrence relation homomorphisms}

We have seen that our definition of a recurrence relation homomorphism is more natural than Definition \ref{defrrh} given in \cite{gandhi} and in the remaining part of the paper we shall derive certain interesting results and consequences of this definition. We begin with the following definition.

\begin{dfn}
Let $a\,:\,\mathbb{N}\rightarrow X$ be a $k^{th}$ order of
recurrence relation on $X$ defined by the
map $f\,:\,X^k\rightarrow X$ with $k\geq 1$ and
initial values $a_1,\ldots,a_k$ given. $a$ is periodic modulo a
positive integer $m$ if we can find at least a non-zero positive integer
$\ell(m)$ such that for all $n\in\mathbb{N}$ $
[a_n]_m=[a_{n+\ell(m)}]_m.
$
\end{dfn}

\begin{rem}
The definition above implies that if $a$ is periodic modulo a
positive integer $m$, then we can find at least a non-zero positive integer
$\ell(m)$ such that for all $j,n\in\mathbb{N}$ $
[a_n]_m=[a_{n+j\ell(m)}]_m.
$
\end{rem}

\begin{thm}\label{t2.3}
Let $X$ a (commutative) ring  where an equivalence relation $\sim$ can
be defined so that the canonical surjection $X\rightarrow X/\!\sim$ 
is a surjective morphism of rings. Let $a\,:\,\mathbb{N}\rightarrow X$ be a $k^{th}$ order of
recurrence relation on $X$ defined by a linear
map $f\,:\,X^k\rightarrow X$ with $k\geq 1$ and
initial values $a_1,\ldots,a_k$ given. If
$$
[a_{1+\ell(m)}]_m=[a_1]_m,
$$
$$
\vdots
$$
$$
[a_{k+\ell(m)}]_m=[a_k]_m,
$$
then $a$ is a periodic sequence modulo $m$.
\end{thm}

\begin{proof}
Let us prove the theorem by induction in the case where
$X=\mathbb{Z}$. The generalization of that is trivial. 

In the theorem, we consider a sequence $a$ which is $k^{th}$ order of
recurrence relation on a set $X$ defined by the linear
map $f\,:\,X^k\rightarrow X$ with $k\geq 1$ such that
$[a_{j+\ell(m)}]_m=[a_j]_m$ with $j=1,\ldots,k$. Let us assume that $[a_{j+\ell(m)}]_m=[a_j]_m$ with $j=k+1,\ldots,n$
and $n>k$. We have
$$
[a_{n+1+\ell(m)}]_m
=\psi_f([a_{n-k+1+\ell(m)}]_m,\ldots,[a_{n+\ell(m)}]_m).
$$
Since the numbers $n-k+i$ with $i\in[[ 1,k]]$ are less
than $n$  we get
$$
[a_{n+1+\ell(m)}]_m=\psi_f([a_{n-k+1}]_m,\ldots,[a_n]_m)=[a_{n+1}]_m.
$$
This completes the rest of the proof.
\end{proof}

\begin{prop}\label{p2.4}
Let $i,j$ be two non-zero positive integers. If $a$ is a strong
divisibility sequence which is periodic
modulo $m$, then a period $\ell(m)$ of the sequence $a$
modulo $m$ satisfies $
[a_{(i+\ell(m),j)}]_m=[w]_m[a_{(i,j)}]_m
$
with $w\in\mathbb{Z}$.
\end{prop}

\begin{proof}
Let $i,j$ two non-zero positive integers.
If $a$ is a strong divisibility sequence which is periodic modulo $m$
with period $\ell(m)>0$, then we have
$$
[a_{(i+\ell(m),j)}]_m=[(a_{i+\ell(m)},a_j)]_m.
$$
Moreover, there exist two integers $x,y$ such that
$$
(a_{i+\ell(m)},a_j)=xa_{i+\ell(m)}+ya_j.
$$
Since $[a_{i+\ell(m)}]_m=[a_i]_m$, we can find an integer $k$ such
that $a_{i+\ell(m)}=a_i+km$. It implies
\begin{align}
(a_{i+\ell(m)},a_j) &= xa_i+ya_j+xkm\nonumber\\
 &\equiv xa_i+ya_j\pmod m .\nonumber
 \end{align}
 Thus, $[(a_{i+\ell(m)},a_j)]_m=[xa_i+ya_j]_m.$ Since $(a_i,a_j)$ divides any linear combination of $a_i,a_j$, there
exists an integer $w$ such that $xa_i+ya_j=w(a_i,a_j)$. We thus have
$$
[(a_{i+\ell(m)},a_j)]_m=[w]_m[(a_i,a_j)]_m=[w]_m[a_{(i,j)}]_m.
$$
\end{proof}

\begin{prop}\label{p2.5}
Let $i,j$ be two non-zero positive integers. If $a$ is a strong
divisibility sequence which is periodic
modulo $m$, then for any given $n\in\mathbb{Z}$, there exists
$w_n\in\mathbb{Z}$ such that
$$
[a_{n(i,j)}]_m=[w_n]_m[a_{(i,j)}]_m.
$$
\end{prop}

\begin{proof}
Let $i,j$ two non-zero positive integers.
Let $a$ be a strong divisibility sequence which is periodic modulo $m$
with period $\ell(m)>0$. Then, there exist three integers $x,y,z$ such
that
\begin{align}
(i+\ell(m),j) &= x(i+\ell(m))+yj\nonumber\\
 &= xi+yj+x\ell(m)\nonumber\\
 &= z(i,j)+x\ell(m).\nonumber
 \end{align}
Since $(i+\ell(m),j)>0$, if $z>0$, then it follows that
$$
[a_{(i+\ell(m),j)}]_m=[a_{z(i,j)+x\ell(m)}]_m=[a_{z(i,j)}]_m.
$$
Or, from Proposition \ref{p2.4}, there exists an
integer $w_z$ such that
$[(a_{i+\ell(m)},a_j)]_m=[w_z]_m[a_{(i,j)}]_m$. Therefore, we deduce
that ($z>0$)
$$
[a_{z(i,j)}]_m=[w_z]_m[a_{(i,j)}]_m.
$$
The case for $z<0$ can now be easily verified from the previous case.
\end{proof}

\begin{rem}
If $i,j$ are two non-zero integers such that
$(i+\ell(m),j)=(i,j)+\ell(m)$ with $\ell(m)$ a period of a sequence
$a$ modulo $m$, then $(i,j)$ divides a multiple of
$\ell(m)$. Moreover, in this case,
we have $[a_{(i+\ell(m),j)}]_m=[a_{(i,j)+\ell(m)}]_m=[a_{(i,j)}]_m$.
\end{rem}

We now find an algorithm to find a period of a sequence modulo a non-zero positive integer $m$.

Let $i,j,h$ be three non-zero positive integers such that
$(i,j)=g$ and $(h,j)=1$ with $gh>i$. If $a$ is a strong divisibility sequence
which is periodic modulo $m$, then the non-zero
positive number $t=gh-i$ satisfies $
[a_{(i+t,j)}]_m=[a_g]_m.$ Since $gh>i$, then $t=gh-i>0$. Thus, we can try numbers like $t$ in
order to find a period of a strong divisibility sequence $a$ which is
periodic modulo $m$.

For instance, let us consider the Fibonacci sequence
$(F_n)_{n\in\mathbb{N}}$ (in this
case, we have $k=2$, which refers to a $2$th order of recurrence relation on a set $X$ defined by a (linear) map). Let
$5q+2$ be a prime with $q$ an odd positive integer. We take $i=5q+2$
and $j=5q+3$. Since
$i,j$ are two consecutive integers, the numbers $i,j$ are relatively
prime $(i,j)=g=1$. Moreover, taking $h=i+2j=15q+8$, we can notice that
$3j-h=1$. So, from Bezout's identity, we have $(h,j)=1$. We have
$gh=15q+8>i$. The number
$t=gh-i$ is given by $t=2(5q+3)$. Or, $2(5q+3)$ is
a period of the Fibonacci sequence modulo $5q+2$ with $q$ an odd
positive integer. Thus, the algorithm allows to get a period of the Fibonacci
sequence modulo $5q+2$ with $q$ an odd positive integer.

The above result was also found in \cite{manjil} by independent methods.

We are now ready to prove and discuss a few more general results in the remainder of this section.

\begin{thm}
Let $a\,:\,\mathbb{N}\rightarrow X$ be a $k^{th}$ order of
recurrence relation on $X$ defined by a linear
map $f\,:\,X^k\rightarrow X$ with $k\geq 1$ and
initial values $a_1,\ldots,a_k$ given.
The sequence $a$ is periodic modulo $m$ with period $\ell(m)>k-1$
if for all $i\in[[ 1,k]]$
$$
[f_i]_m=[a_i]_m,
$$
$$
[a_{2i+\ell(m)-k-1}]_m=[1]_m,
$$
and
$$
{\displaystyle\sum_{j\in[[ 1,k]]-\left\{i\right\}}}[f_j]_m\cdot
[a_{i+\ell(m)-k+j-1}]_m=[0]_m.
$$
\end{thm}

\begin{proof}
We can notice that since $\ell(m)>k-1$, we have $\ell(m)>k-i$ for all
$i\in[[ 1,k]]$. Thus
$$
a_{i+\ell(m)}=f(a_{i+\ell(m)-k},\ldots,a_{i+\ell(m)-1})
={\displaystyle\sum^k_{j=1}}f_j\cdot a_{i+\ell(m)-k+j-1}.
$$
So,
\begin{align}
[a_{i+\ell(m)}]_m &= {\displaystyle\sum^k_{j=1}}[f_j]_m\cdot
[a_{i+\ell(m)-k+j-1}]_m\nonumber\\
 &= [f_i]_m\cdot [a_{2i+\ell(m)-k-1}]_m
+{\displaystyle\sum_{j\in[[ 1,k]]-\left\{i\right\}}}[f_j]_m\cdot
[a_{i+\ell(m)-k+j-1}]_m\nonumber\\
 &= [a_i]_m.\nonumber
 \end{align}
Since $i$ is any number of the set $[[ 1,k]]$, from
Theorem \ref{t2.3}, we conclude that $\ell(m)$ is a period of the
sequence $a$.
\end{proof}

\begin{thm}\label{t2.9}
Let $a\,:\,\mathbb{N}\rightarrow X$ be a $k^{th}$ order of
recurrence relation on $X$ defined by a linear
map $f\,:\,X^k\rightarrow X$ with $k\geq 2$ and
initial values $a_1,a_2,\ldots,a_k$ given. If $a$ is a periodic
sequence modulo $m$ with period $\ell(m)$, then
$$
[a_k]_m=[f_1]_m[a_{\ell(m)}]_m+{\displaystyle\sum^k_{i=2}}[f_i]_m[a_{i-1}]_m.
$$
\end{thm}

The proof is an easy application of Theorem \ref{t2.3}, so for the sake of brevity we shall omit it here.

\begin{rem}\label{r2.10}
Theorem \ref{t2.9} allows us to find in an algorithmic way, a period
of sequence $a$ modulo some positive integer $m\geq 1$. Indeed, the
residue class
$[r_{\ell(m)}]_m$ of $a_{\ell(m)}$ modulo a positive integer
$m \geq 1$ such that $r_{\ell(m)}$ belongs to $[[0,m-1]]$, can be found by
solving in the ring $\mathbb{Z}/m\mathbb{Z}$, the diophantine
equation
$$
[a_k]_m=[f_1]_m[a_{\ell(m)}]_m+
{\displaystyle\sum^k_{i=2}}[f_i]_m[a_{i-1}]_m
$$
where the unknown is $[a_{\ell(m)}]_m$ and $[a_i]_m$ with
$i=1,2,\ldots,k$ such that $k\geq 2$ as well as $m$ are given. 
\end{rem}

\begin{thm}
Let $a\,:\,\mathbb{N}\rightarrow X$ be a $k^{th}$ order of
recurrence relation on $X$ defined by a linear
map $f\,:\,X^k\rightarrow X$ with $k\geq 1$ and
initial values $a_1,\ldots,a_k$ given. Then, we have ($k\geq i\geq 1$)
$$
a_{k+i}={\displaystyle\sum^i_{m=1}}C_{k,i-m+1}
{\displaystyle\sum^k_{j=m}}f_{j-m+1}a_j,
$$
with the sequence $(C_{k,n})$ defined by ($k\geq 1$)
$$
C_{k,1}=1,
$$
and ($n\in[[ 2,k]]$ with $k\geq 2$)
$$
C_{k,n}={\displaystyle\sum^{n-1}_{j=1}}f_{k-j+1}C_{k,n-j}.
$$
\end{thm}

\begin{proof}
We can notice that for $i\geq 1$,
$$
a_{k+i}={\displaystyle\sum^k_{j=1}}f_ja_{i+j-1}
={\displaystyle\sum^{k-i+1}_{j=1}}f_ja_{i+j-1}
+f_{k-i+2}a_{k+1}+\ldots+f_ka_{k+i-1}
$$
So for $2\leq i\leq k$, it gives
\begin{align} 
a_{k+i} &= {\displaystyle\sum^{k-i+1}_{j=1}}f_ja_{i+j-1}
+{\displaystyle\sum^{i-1}_{j=1}}f_{k-i+j+1}a_{k+j}
={\displaystyle\sum^{i-1}_{j=1}}f_{k-i+j+1}a_{k+j}
+{\displaystyle\sum^{k-i+1}_{j=1}}f_ja_{i+j-1}\nonumber\\
 &= {\displaystyle\sum^{i-1}_{j=1}}f_{k-i+j+1}a_{k+j}
+{\displaystyle\sum^k_{j=i}}f_{j-i+1}a_j\nonumber
\end{align} where we make the change of label $j\rightarrow l=i-1+j$ and
afterwards we renamed $l$ by $j$ in the discrete sum
$\sum^{k-i+1}_{j=1}f_ja_{i+j-1}$.

Let us prove the theorem by finite induction on the integer $i$ (see
\cite{chartrand} p.146, exercise 27). We have
$$
a_{k+1}={\displaystyle\sum^k_{j=1}}f_ja_j
=C_{k,1}{\displaystyle\sum^k_{j=1}}f_ja_j
={\displaystyle\sum^1_{m=1}}C_{k,2-m}
{\displaystyle\sum^k_{j=m}}f_{j-m+1}a_j.
$$
Let us assume that for an integer $1\leq i<k$, we have
$a_{k+i}={\displaystyle\sum^i_{m=1}}C_{k,i-m+1}
{\displaystyle\sum^k_{j=m}}f_{j-m+1}a_j$. Using the formula of
$a_{k+i}$ above and the assumption, we have
\begin{align}
a_{k+i+1} &= {\displaystyle\sum^i_{j=1}}f_{k-i+j}a_{k+j}
+{\displaystyle\sum^k_{j=i+1}}f_{j-i}a_j\nonumber\\
 &= {\displaystyle\sum^i_{j=1}}f_{k-i+j}
{\displaystyle\sum^j_{m=1}}C_{k,j-m+1}
{\displaystyle\sum^k_{l=m}}f_{l-m+1}a_l
+{\displaystyle\sum^k_{j=i+1}}f_{j-i}a_j.\nonumber
\end{align}
Or,
$$
{\displaystyle\sum^i_{j=1}}f_{k-i+j}{\displaystyle\sum^j_{m=1}}C_{k,j-m+1}
{\displaystyle\sum^k_{l=m}}f_{l-m+1}a_l
={\displaystyle\sum^i_{j=1}}f_{k-j+1}
{\displaystyle\sum^{i-j+1}_{m=1}}C_{k,i+1-m+1-j}
{\displaystyle\sum^k_{l=m}}f_{l-m+1}a_l
$$ where we made the change of label $j\rightarrow t=i-j+1$ and
afterwards we renamed $t$ by $j$.

We can notice that for fixed $m$, $j$ runs from $1$ to $i-m+1$ since
from the definition of the sequence $(C_{i,n})$, the label $i+1-m+1-j$
should be greater than 1. Since the minimum value of $m$ is 1 and the
maximum value of $m$ is $i$, permuting the discrete sums over $j,m$,
it results that
\begin{align}
{\displaystyle\sum^i_{j=1}}f_{k-i+j}{\displaystyle\sum^j_{m=1}}C_{k,j-m+1}
{\displaystyle\sum^k_{l=m}}f_{l-m+1}a_l &= {\displaystyle\sum^i_{m=1}}{\displaystyle\sum^{i-m+1}_{j=1}}
f_{k-j+1}C_{k,i+1-m+1-j}{\displaystyle\sum^k_{l=m}}f_{l-m+1}a_l\nonumber\\
 &= {\displaystyle\sum^i_{m=1}}C_{k,i+1-m+1}{\displaystyle\sum^k_{l=m}}f_{l-m+1}a_l.\nonumber
\end{align}
So, we have
\begin{align}
a_{k+i+1} &= {\displaystyle\sum^i_{m=1}}C_{k,i+1-m+1}{\displaystyle\sum^k_{l=m}}f_{l-m+1}a_l
+{\displaystyle\sum^k_{j=i+1}}f_{j-i}a_j\nonumber\\
 &= {\displaystyle\sum^i_{m=1}}C_{k,i+1-m+1}
{\displaystyle\sum^k_{l=m}}f_{l-m+1}a_l
+C_{k,1}{\displaystyle\sum^k_{j=i+1}}f_{j-(i+1)+1}a_j\nonumber\\
 &= {\displaystyle\sum^{i+1}_{m=1}}C_{k,i+1-m+1}
{\displaystyle\sum^k_{l=m}}f_{l-m+1}a_l.\nonumber
\end{align}
Thus the proof of the theorem is complete by induction.
\end{proof}

\begin{cor}
Let $a\,:\,\mathbb{N}\rightarrow X$ be a $k^{th}$ order of
recurrence relation on $X$ defined by a linear
map $f\,:\,X^k\rightarrow X$ with $k\geq 2$ and
initial values $a_1,a_2,\ldots,a_k$ given. Then, we have ($k>i\geq 1$)
$$
a_{k+i}={\displaystyle\sum^i_{m=1}}a_m
{\displaystyle\sum^m_{j=1}}f_jC_{k,i-m+j}+{\displaystyle\sum^k_{m=i+1}}a_m
{\displaystyle\sum^i_{j=1}}f_{m-i+j}C_{k,j}.
$$
\end{cor}

\begin{proof}
From the theorem above, we have for $k>i\geq 1$, 
\begin{align}
a_{k+i} &= {\displaystyle\sum^i_{m=1}}C_{k,i-m+1}
{\displaystyle\sum^k_{j=m}}f_{j-m+1}a_j\nonumber\\
 &= C_{k,i}
{\displaystyle\sum^k_{j=1}}f_ja_j+C_{k,i-1}
{\displaystyle\sum^k_{j=2}}f_{j-1}a_j
+\ldots+C_{k,1}
{\displaystyle\sum^k_{j=i}}f_{j-i+1}a_j\nonumber\\
 &= {\displaystyle\sum^i_{m=1}}a_m\left[
C_{k,i}f_m+\ldots+C_{k,i-m+1}f_1\right]+{\displaystyle\sum^k_{m=i+1}}a_m\left[
C_{k,i}f_m+\ldots+C_{k,1}f_{m-i+1}\right]\nonumber\\
 &= {\displaystyle\sum^i_{m=1}}a_m
{\displaystyle\sum^m_{j=1}}f_jC_{k,i-m+j}+{\displaystyle\sum^k_{m=i+1}}a_m
{\displaystyle\sum^i_{j=1}}f_{m-i+j}C_{k,j}.\nonumber
\end{align}
\end{proof}

Thus, a generic term $a_{k+i}$ with $k>i\geq 1$ of a sequence $a$ which
is a $k^{th}$ order of
recurrence relation on $X$ defined by a linear
map $f\,:\,X^k\rightarrow X$ with $k\geq 2$ and
initial values $a_1,a_2,\ldots,a_k$ given, can be
rewritten as
$$
a_{k+i}={\displaystyle\sum^k_{m=1}}(M_k)_{i,m}a_m,
$$
with $M_k$ defined by
$$
(M_k)_{i,m}=\left\{\begin{array}{cc}
{\displaystyle\sum^m_{j=1}}f_jC_{k,i-m+j} & 1\leq m\leq i,\\
{\displaystyle\sum^i_{j=1}}f_{m-i+j}C_{k,j} & i<m\leq k.
\end{array}
\right.
$$
This formula implies that for $1\leq l(m) <k$ we have
$$
[a_{k+\ell(m)}]_m={\displaystyle\sum^k_{i=1}}[(M_k)_{\ell(m),i}]_m[a_i]_m=[a_k]_m.
$$

Thus, we obtain a diophantine equation in the ring
$\mathbb{Z}/m\mathbb{Z}$ where the residue class
$[(r_k)_{\ell(m),i}]_m$ of $(M_k)_{\ell(m),i}$ modulo a (non-zero) positive
integer $m$ such that the numbers
$(r_k)_{\ell(m),i}$ belong to $[[0,m-1]]$, are the
unknowns and $[a_i]_m$ with $i=1,\ldots,k$ as well as $m$ are given. Solving
this equation, it allows to determine a period $\ell(m)$ of the
sequence $a$ modulo $m$. Indeed, since all the coefficients of matrix
$M_k$ can be computed by the formula above, it suffices to compare
numbers $(r_k)_{\ell(m),i}+tm$ with $t$ an integer with the numbers
$(M_k)_{l,i}$ with $l$ a non-zero positive integer. A value of
label $l$ for which $(r_k)_{\ell(m),i}+tm=(M_k)_{l,i}$ whatever $i\in
[[1,k]]$ corresponds to a value of a period $\ell(m)$ of sequence $a$
modulo $m$.

We can notice that if a sequence $a$ which
is a $k^{th}$ order of
recurrence relation on $X$ defined by a linear
map $f\,:\,X^k\rightarrow X$ with $k\geq 1$ and
initial values $a_1,\ldots,a_k$ given, is a strong divisibility
sequence, then from the associative property of the $GCD$ operation,
we have ($n\geq 1$ and $s_l\geq 1$ with $l\in[[ 1,n]]$)
$$
(a_{s_1},\ldots,a_{s_n})=a_{(s_1,\ldots,s_n)}.
$$

We recall the following easy exercise from \cite{apostol} without proof.

\begin{prop}
Given two positive integers $x$ and $y$, let $m,n$ two
positive integers such that $m=ax+by$ and $n=cx+dy$ with $ad-bc=\pm
1$. Then we have $
(m,n)=(x,y).$
\end{prop}

We generalize the above as follows

\begin{prop}
Let $n$ be a positive integer which is greater than $2$.
Given $n$ positive integers $x_1,x_2,\ldots,x_n$, let
$y_1,y_2,\ldots,y_n$ be $n$ positive integers such that ($i=1,2,\ldots,n$)
$$
y_i={\displaystyle\sum^n_{j=1}}A_{i,j}x_j,
$$
with $\det(A)=\pm 1$. Then we have
$$
(y_1,y_2,\ldots,y_n)=(x_1,x_2,\ldots,x_n).
$$
\end{prop}

\begin{proof}
Let $g=(x_1,x_2,\ldots,x_n)$ and $G=(y_1,y_2,\ldots,y_n)$. So, there
exists $2n$ integers, say $$u_1,u_2,\ldots,u_n,U_1,U_2,\ldots,U_n$$
such that
$$
g=u_1x_1+u_2x_2+\ldots+u_nx_n
$$
and
$$
G=U_1y_1+U_2y_2+\ldots+U_ny_n.
$$

Let $d$ a common divisor of $x_1,x_2,\ldots,x_n$. From the linearity property of divisibility, since $d|x_i$ with
$i=1,2,\ldots,n$,
$d|y_i$ with $i=1,2,\ldots,n$ and so $d|G$. In particular, $g|G$. Let $D$ a common divisor of $y_1,y_2,\ldots,y_n$.

If $A$ is a $n\times n$ square matrix whose determinant is non-zero
($\det(A)=\pm 1$ and so $rank(A)=n$), then the linear system of equations
$y_i={\displaystyle\sum^n_{j=1}}A_{i,j}x_j$ with $i=1,2,\ldots,n$
is a Cramer linear system of $n$ equations
which has a unique solution given by the $n$-tuple
$(x_1,x_2,\ldots,x_n)$ such that ($i=1,2,\ldots,n$)
$$
x_i=\frac{\Delta_i(A)}{\det(A)}=\pm\Delta_i(A),
$$
where $\Delta_i(A)$ is the determinant of the $n\times n$ square
matrix which is obtained
from the matrix $A$ by replacing the $i^{th}$ column of $A$ by the column
$\left(\begin{array}{c}y_1\\ y_2\\ \vdots\\ y_n\end{array}\right)$.

From the linearity property of divisibility, since $D|y_i$ with
$i=1,2,\ldots,n$, $D|x_i$ with $i=1,2,\ldots,n$ and so $D|g$. In
particular, $G|g$.

From $g|G$ and $G|g$, since $g$ and $G$ are positives, it results that
$g=G$.
\end{proof}

\begin{rem}
This property can be extended to the case where the determinant of the
matrix $A$ is a common divisor of the numbers
$\Delta_1(A),\Delta_2(A),\ldots,\Delta_n(A)$.
\end{rem}

\begin{rem}
If a sequence $a$ which
is a $k^{th}$ order of
recurrence relation on $X$ defined by a linear
map $f\,:\,X^k\rightarrow X$ with $k\geq 2$ and
initial values $a_1,a_2,\ldots,a_k$ given, is a strong divisibility
sequence, since $a_{k+i}$ with $i\geq 1$ is a linear combination of
$a_1,a_2,\ldots,a_k$, if the determinant of the $k\times k$
square matrix $((M_k)_{i,m})$ with $1\leq i\leq k$ and $1\leq m\leq k$
which we denote simply by $M_k$ when there is no ambiguity (the matrix
elements $(M_k)_{i,m}$ was defined previously for $1\leq i<k$ and the
matrix elements $(M_k)_{k,m}$ can be determined from the definition of
sequence $a$), is either $\pm 1$ or a common divisor of the numbers
$\Delta_1(M_k),\Delta_2(M_k),\ldots,\Delta_k(M_k)$, then we have
$$
(a_{k+1},a_{k+2},\ldots,a_{2k})=(a_1,a_2,\ldots,a_k)=a_{(1,2,\ldots,k)}=a_1.
$$
\end{rem}

\section*{Acknowledgments}

The authors would like to thank the annonymous referees and the editor for helpful comments. The second author was supported by DST INSPIRE Scholarship 422/2009 from Department of Science and Technology, Government of India.



\end{document}